\documentclass{amsart}

\usepackage{amsmath}
\usepackage{amssymb}
\usepackage{color}
\usepackage{fullpage}
\usepackage{mathrsfs}
\usepackage{hyperref}
\usepackage{enumitem}
\usepackage{stmaryrd}
\usepackage[all]{xy}
\usepackage{verbatim}

\numberwithin{equation}{section}
\theoremstyle{plain}
  \newtheorem{theorem}[equation]{Theorem}
    
  \newtheorem{proposition}[equation]{Proposition}
  \newtheorem{lemma}[equation]{Lemma}
  \newtheorem{corollary}[equation]{Corollary}

\theoremstyle{definition}
  \newtheorem{definition}[equation]{Definition}

\theoremstyle{remark}
  \newtheorem{example}[equation]{Example}
  \newtheorem{remark}[equation]{Remark}

\newcommand{\enumeratea}{\thispagestyle{empty} \begin{enumerate}[label=\arabic*., ref=\arabic*, itemsep=1em, leftmargin=1em]}
\newcommand{\enumerateb}{\begin{enumerate}[label=\alph*., ref=\alph*, topsep = 3pt, leftmargin=1.25em, itemsep=.5em, listparindent=2em]}

\newcommand{\QQ}{\mathbb{Q}}

\newcommand{\ZZ}{\mathbb{Z}}

\DeclareMathOperator{\TC}{TC}
\DeclareMathOperator{\THH}{THH}
\DeclareMathOperator{\TP}{TP}
\DeclareMathOperator{\HH}{HH}
\DeclareMathOperator{\HC}{HC}
\DeclareMathOperator{\HP}{HP}

\newcommand{\Rpp}{R \mathbin{\!/\mkern-6mu/}\! p}
\newcommand{\Zpp}{\mathbb{Z} \mathbin{\!/\mkern-6mu/}\! p}
\newcommand{\Rppp}{R_{(p)} \mathbin{\!/\mkern-6mu/}\! p}
\usepackage{graphicx}

\begin{document}

\title{On the cyclic homology of certain universal differential graded algebras}

\author{Christopher Davis}
\author{Julius Frank}
\author{Irakli Patchkoria}

\maketitle

\begin{abstract}
Let $p$ be an odd prime and $R$ a $p$-torsion-free commutative $\ZZ_{(p)}$-algebra. We compute the periodic cyclic homology over $R$ of the universal differential graded algebra $\Rpp$ which is obtained from $R$ by universally killing $p$. We furthermore compute the cyclic and negative cyclic homologies of $\Rpp$ over $R$ in infinitely many degrees. 
\end{abstract}

\section{Introduction}\label{introduction section}

For a fixed prime $p$, we can consider the ring $\ZZ[u]/(u(u-p))$. It is an interesting problem to compute the algebraic $K$-theory $K(\ZZ[u]/(u(u-p)))$  of this ring. It follows from a result of Land and Tamme \cite[Theorem 1.1 and Example 4.31]{LT} that there is a homotopy pullback square of algebraic $K$-theory spectra
\[\xymatrix{K(\ZZ[u]/(u(u-p))) \ar[d] \ar[r] & K(\ZZ) \ar[d] \\ K(\ZZ) \ar[r] & K(\Zpp),}\]
where $\Zpp$ is the universal differential graded algebra with one generator $x$ in degree $1$ satisfying $dx=p$. 

The differential graded algebra $\Zpp$ has a different behaviour depending on whether the prime $p$ is even or odd. For $p=2$, it follows from unpublished results of Krause and Nikolaus and \cite[Example 4.32]{LT} that the algebra $\Zpp$ is formal over the sphere spectrum. For $p$ odd, the results of the second author \cite{Frank} imply that $\Zpp$ is not formal as an $E_1$-ring spectrum and hence also not as a differential graded algebra. 

More generally, one can consider any $p$-torsion-free commutative $\ZZ_{(p)}$-algebra $R$ and the differential graded algebra $\Rpp$. The underlying graded algebra of $\Rpp$ is $R[x]$ such that $x$ is in degree $1$ and $dx=p$. It follows from the Dundas-McCarthy theorem \cite{Dund, Mc} that the commutative diagram
\[ \xymatrix{K(\Rpp) \ar[r]  \ar[d] & \TC(\Rpp) \ar[d] \\ K(R/p) \ar[r] & \TC(R/p)  }\]
is a homotopy pullback after $p$-adic completion, where $\TC$ denotes the topological cyclic homology and the horizontal maps are given by the cyclotomic trace map (see \cite{BokHsiaMad}). If one assumes the knowledge of $K(R/p)$ and $\TC(R/p)$, then using the latter square in order to compute $K(\Rpp)$, it suffices to compute $\TC(\Rpp)$. One can try to calculate $\TC(\Rpp)$ using the approach of Nikolaus-Scholze \cite{NS}. Recall that the topological Hochschild homology spectrum $\THH(\Rpp)$ has a circle action, and one denotes the homotopy fixed points with respect to this action by $\TC^{-}(\Rpp)$ and the Tate construction by $\TP(\Rpp)$. Using \cite[Proposition II.1.9 and Lemma II.4.2.1]{NS}, we know that after $p$-completion there exists a fiber sequence of spectra 
\[\xymatrix{\TC(\Rpp) \ar[r] & \TC^{-}(\Rpp) \ar[r]^-{can-\varphi} & \TP(\Rpp),}\] where $can$ is the canonical map and $\varphi$ the Frobenius.  
Thus in order to compute the topological cyclic homology $\TC(\Rpp)$, one should compute $\TC^{-}(\Rpp)$ and $\TP(\Rpp)$ as well as the maps $can$ and $\varphi$.

The ultimate goal of this project is to compute the $p$-completion of the spectrum $K(\Rpp)$ for $R$ a $p$-torsion-free perfectoid ring in the sense of \cite{BMS1}. In this case it follows by \cite{CMM} and the Dundas-McCarthy theorem \cite{Dund, Mc}, that the $p$-completion of $K(\Rpp)$ is the connective cover of the $p$-completion of $\TC(\Rpp)$. Hence we need to compute $\TC^{-}(\Rpp)$ and $\TP(\Rpp)$ and eventually also $\TC(\Rpp)$ for $R$ a $p$-torsion-free perfectoid ring. We do not recall here perfectoid rings since the definition of these will not be relevant in this paper. However, we do wish to recall that by \cite{BMS2},  for any perfectoid ring $R$, there is a circle equivariant fiber sequence after $p$-completion: 
\[\THH(\Rpp)[2] \xrightarrow{}  \THH(\Rpp) \to  \HH^R(\Rpp),\]
where $\HH^R(\Rpp)$ denotes the Hochschild homology of $\Rpp$ over $R$. By applying either the homotopy fixed points or Tate construction, one gets fiber sequences after $p$-completion:
\[\xymatrix{\TC^{-}(\Rpp)[2] \to \TC^{-}(\Rpp) \to \HC^{R,-}(\Rpp),\:\:\:\; \TP(\Rpp)[2] \to \TP(\Rpp) \to \HP^R(\Rpp),}\]
where $\HC^{R,-}(\Rpp)$ and $\HP^R(\Rpp)$ denotes the negative and periodic cyclic homology over $R$, respectively \cite[Chapter 5]{Lod}. 

The goal of this paper is to compute $\HC^{R,-}(\Rpp)$ and $\HP^R(\Rpp)$ for any $p$-torsion-free commutative $\ZZ_{(p)}$-algebra $R$ which is not necessarily perfectoid. Having the above fiber sequences in mind, the hope is that we can solve extension problems and compute $\TC^{-}(\Rpp)$ and $\TP(\Rpp)$. 

For simplicity we denote by $\HC_i(\Rpp)$, $\HC^{-}_i(\Rpp)$ and $\HP_i(\Rpp)$, the cyclic, negative cyclic and periodic cyclic homology modules of $\Rpp$ over $R$, respectively. We can now formulate the main results of this paper:

\begin{theorem} Let $p$ be an odd prime and $R$ a $p$-torsion-free commutative $\ZZ_{(p)}$-algebra. Then $\HP_i(\Rpp)=0$ for $i$ odd and for any even $i$, the $R$-module $\HP_i(\Rpp)$ is isomorphic to
\[R^{\wedge} \times R/1 \times R/3 \times R/5 \times \cdots,\]
where $R^{\wedge}$ denotes the $p$-adic completion of $R$. 
\end{theorem}

Note that all the factors $R/n$ with $n$ coprime to $p$ vanish, and more generally, for any odd positive integer $n$, we have $R/n\cong R/p^{\nu_p(n)}$, where $\nu_p$ denotes the $p$-adic valuation. 

To formulate the next result on the cyclic homology we need to define the following numbers: Let $A_1 := p$, and for each odd integer $j \geq 3$, we recursively define $A_j \in \QQ$ as $A_j := \frac{p^2 A_{j-2}}{j}$. Corollary \ref{AB in Zp} below shows that these numbers belong to $\ZZ_{(p)}$ and hence the $p$-adic valuations $a_i=\nu_p(A_i)$ are non-negative integers. 

\begin{theorem} Let $p$ be an odd prime and $R$ a $p$-torsion-free commutative $\ZZ_{(p)}$-algebra. Then $\HC_i(\Rpp)=0$ for $i$ odd and $\HC_0(\Rpp)=R/p$. There furthermore exists an infinite set $Z$ of positive even integers, such that for any $i \in Z$, the $R$-module $\HC_i(\Rpp)$ is isomorphic to
\[
R/p^{a_{i-1}+2} \oplus R/1 \oplus R/3 \oplus \cdots \oplus R/(i-1).
\]
The set $Z$ contains all even numbers of the form $\frac{p^a\pm 1}{2}+1$ for $a>0$. 
\end{theorem}

Finally, we have the following result calculating the negative cyclic homology:

\begin{theorem} Let $p$ be an odd prime and $R$ a $p$-torsion-free commutative $\ZZ_{(p)}$-algebra. Then $\HC^{-}_i(\Rpp)=0$ for $i$ odd and for any non-positive even $i$, the $R$-module $\HC^{-}_i(\Rpp)$ is isomorphic to
\[R^{\wedge} \times R/1 \times R/3 \times R/5 \times \cdots.\]
There furthermore exists an infinite set $\overline{Z} \subset Z$, such that for any $i \in \overline{Z}$, the $R$-module $\HC^{-}_i(\Rpp)$ is isomorphic to
\[R^{\wedge} \times R/(i-1) \times R/(i+1) \times R/(i+3) \times \cdots.\]
The set $\overline{Z}$ contains all even numbers of the form $\frac{p^a\pm 1}{2}+1$ for $a>0$. 
\end{theorem}

Additionally, we also compute the values of $\HC_{i-2}(\Rpp)$ for $i \in \overline{Z}$ and $i>2$. This is a consequence of the proof of the latter theorem from which it follows that for any $i \in \overline{Z}$ and $i>2$,  we have an isomorphism 
\[\HC_{i-2}(\Rpp) \cong R/p^{a_{i-1}} \oplus R/1 \oplus R/3 \oplus \cdots \oplus R/(i-3).\]

The values of $\HC_i(\Rpp)$ and $\HC^{-}_i(\Rpp)$ for a general $i$ remain still open. However, in the final section of this paper we estimate the sizes of the sets $Z$ and $\overline{Z}$ and conclude that for large primes they get asymptotically close to the set of all positive even integers. 

\section*{Acknowledgements}

We would like to thank John Jones and Markus Land for helpful conversations.

\section{Setup} \label{setup section}
Fix an odd prime $p$ and let $R$ be a commutative ring without $p$-torsion. We do not yet require $R$ to be a $\ZZ_{(p)}$-algebra.
\begin{definition}
  We construct a differential graded algebra $\Rpp$ 
  over $R$ as follows: It has as its underlying graded algebra $R[x]$ with $x$ in degree $1$, and the differential is given as $\delta x=p$.
\end{definition}
\begin{remark} 
We collect some immediate observations for this algebra: 
\begin{itemize}
  \item Note that while this is strictly commutative, it is not graded commutative. 
  \item The differential in higher degrees is 
  \[
    \delta x^n=\begin{cases}
      p x^{n-1} & n \textup{ odd}\\
      0 & n \textup{ even.}
    \end{cases}
  \]
  Because $R$ has no $p$-torsion, the homology is then $H_*(\Rpp)= R/p[u]$, with $u$ in degree $2$.
\end{itemize}

\end{remark}
To justify the notation, we give a more conceptual description of $\Rpp$. Let $T_R:\textup{Ch}_R\to \textup{DGA}_R$ denote the tensor algebra functor that assigns to each chain complex $M$ the free differential graded $R$-algebra $T_R(M)=\bigoplus_{m\geq 0}M^{\otimes_R m}$. 
\begin{proposition}
  There is a pushout square in $\textup{DGA}_R$ given by
  \[
   \xymatrixrowsep{1.5em}
   \xymatrixcolsep{2em}
   \xymatrix{  
   T_R(R) \ar^{p}[r]\ar_{}[d] & R\ar_{}[d]\\
   T_R(CR) \ar_{}[r] & \Rpp,
   }
   \]
   where $CR$ is the cone of $R$ in $\textup{Ch}_R$, the left map is induced by the inclusion $R\hookrightarrow CR$, and the top map is adjoint to the multiplication $p:R\to R$.

\end{proposition}
\begin{proof}
  A morphism $\varphi:\Rpp\to A$ of dg-$R$-algebras is precisely a ring homomorphism $\tilde\varphi:R\to A_0$ and the choice of an $a\in A_1$ such that $\tilde\varphi(p)=\delta a$. By identifying $T_R(CR)\cong  R\langle x, \delta x\rangle$ with $|x|=1$, and $T_R(R)\cong R\langle \delta x\rangle$, this data corresponds precisely to a commutative diagram 
  \[
   \xymatrixrowsep{1.5em}
   \xymatrixcolsep{2em}
   \xymatrix{  
   R \langle \delta x \rangle \ar^{p}[r]\ar_{}[d] & R\ar_{}[d]\\
   R\langle x,\delta x\rangle \ar_{}[r] & A,
   }
   \]
   where $x$ is mapped to $a$.

\end{proof}
\begin{remark}
  This exhibits $\Rpp$ as the $\mathbb{E}_1$-quotient of $R$ with respect to $p$, i.e. $\Rpp$ is the initial $\mathbb{E}_1$-R-algebra whose homology has $p=0$. More precisely, there is a model structure on $\textup{DGA}_R$ with quasi-isomorphisms as weak equivalences, under which the left map is a cofibration \cite{Hin, SchSh}. Hence the square is a homotopy pushout of dg-$R$-algebras. The homotopy theory of dg-$R$-algebras is equivalent to the homotopy theory of $\mathbb{E}_1$-algebras in the derived $\infty$-category $\mathcal{D}(R)$ (see \cite{Shi1}). The homotopy pushout then corresponds to a pushout in the $\infty$-category of $\mathbb{E}_1$-algebras over $R$. 
\end{remark}
   
The aim of this paper is to understand cyclic, negative cyclic, and periodic homology of $\Rpp$ over $R$. 

As a first approach, note that there is a Connes long exact sequences connecting cyclic homology with Hochschild homology: This is obtained by considering the subcomplex of the cyclic bicomplex consisting only of the $0$th and $1$st column which is quasi-isomorphic to the Hochschild complex \cite[2.2.1]{Lod}. The quotient complex then computes cyclic homology shifted by degree $2$. The resulting long exact sequence for any algebra $A$ is therefore
\[
  \cdots \longrightarrow \HH_i(A)\longrightarrow \HC_i(A) \longrightarrow \HC_{i-2}(A) \longrightarrow \cdots
\]
In our setting, $\HC_*$ vanishes in odd degrees, so we obtain for $i>0$ even short exact sequences
\[
  0\longrightarrow R/p^2\longrightarrow \HC_{i}(\Rpp)\longrightarrow \HC_{i-2}(\Rpp) \longrightarrow 0,
\]
and we recover $HC_{0}(\Rpp) \cong R/p$.

Similarly, one gets a long exact sequence for negative cyclic homology:
\[
  \cdots \longrightarrow \HC^-_{i+2}(A)\longrightarrow \HC^-_i(A) \longrightarrow \HH_{i}(A) \longrightarrow \cdots.
\]
Below in our setting, we obtain for even $i>0$ short exact sequences
\[
  0 \longrightarrow \HC^-_{i+2}(\Rpp)\longrightarrow \HC^-_i(\Rpp) \longrightarrow R/p^2 \longrightarrow 0.
\]
and for $i=0$
\[
 0 \longrightarrow \HC^-_{2}(\Rpp)\longrightarrow \HC^-_0(\Rpp) \longrightarrow R/p \longrightarrow 0.
\]

This looks innocent enough, but these extension problems will occupy the rest of this paper.

To get started, we will need a better understanding of the Hochschild bicomplex:
\begin{proposition} \label{Hochschildsmall}
   The Hochschild bicomplex of $\Rpp$ is quasi-isomorphic to 
   \[
   \xymatrixrowsep{1.5em}
   \xymatrixcolsep{2em}
   \xymatrix{  
   & R\ar_{0}[d] & R \ar_{p}[l]\ar_{2}[d] & R \ar_{0}[l]\ar_{0}[d] & R \ar_{p}[l]\ar_{2}[d] & \cdots \ar_{0}[l]\\
   R & R\ar_{p}[l] & R \ar_{0}[l] & R\ar_{p}[l] & R \ar_{0}[l] & \cdots \ar_{p}[l]
   }
   \]
 \end{proposition} 
 \begin{proof}
   Note that the underlying graded algebra of $\Rpp$ is a tensor algebra $T_R(V)$ over the free graded $R$-module $V=\langle x \rangle$ with $x$ in degree $1$. For dg-$R$-algebras $A$ with underlying graded tensor algebra $T_R(V)$, the Hochschild bicomplex has a simplified description, \cite[5.3.8]{Lod}: It is quasi-isomorphic to 
   \[
   \xymatrixrowsep{1.5em}
   \xymatrixcolsep{2em}
   \xymatrix{  
   (A\otimes V)_0 \ar_{b}[d]& (A\otimes V)_1\ar_-{\tilde \delta}[l]\ar_{b}[d] & (A\otimes V)_2 \ar_-{\tilde\delta}[l]\ar_{b}[d] & \cdots \ar_-{\tilde\delta}[l]\\
   A_0 & A_1\ar_-{\delta}[l] & A_2 \ar_-{\delta}[l] & \cdots, \ar_-{\delta}[l]
   }
   \]
   where all tensor products are over $R$, and the maps are as follows: $\delta$ is the differential of $A$ and $b(a\otimes v)= [a,v]$. (Note that this is a graded commutator!) For $\tilde\delta$, consider first the map
   \begin{align*}
     \varphi: A \otimes A & \longrightarrow A \otimes V\\
     a\otimes (v_1\otimes \cdots \otimes v_n) & \longmapsto \sum_{i=1}^n \pm \left( v_{i+1}\otimes\cdots\otimes v_n\otimes a\otimes v_1 \otimes \cdots \otimes v_{i-1}\right) \otimes v_i\\
     a\otimes 1&\longmapsto 0,
   \end{align*}
   with the sign determined by the Koszul convention. Then $\tilde\delta$ is given by
   \[
     \tilde\delta (a\otimes v)= \delta a\otimes v + (-1)^{|a|}\varphi(a\otimes \delta v).
   \]
   In our case, $V=\langle x \rangle$ with $|x|=1$. Hence $A_n=\langle x^n\rangle$, and for $n>0$, $(A\otimes V)_n=\langle x^{n-1}\otimes x\rangle$. One can then easily identify the maps: We already know $\delta$. For $b$, we have
   \[
     b(x^{n-1}\otimes x)=[x^{n-1},x]=x^n - (-1)^{n-1}x^n=\begin{cases}2 x^n & n \textup{ even} \\ 0 & n \textup{ odd,} \end{cases}
   \]
   and for $\tilde\delta$,
   \[
     \tilde\delta (x^{n-1}\otimes x)=\delta x^{n-1}\otimes x + (-1)^{n-1}\varphi(x^{n-1}\otimes 
     p)= \delta x^{n-1}\otimes x =\begin{cases} p x^{n-2}\otimes x & n \textup{ even} \\ 0 & n \textup{ odd.} \end{cases}
   \]
 \end{proof}

 \begin{corollary} \label{cor: Hochschild}
  The total complex of the bicomplex of Proposition \ref{Hochschildsmall} that computes Hochschild homology is given by
  \[
  \xymatrixrowsep{1.5em}
  \xymatrixcolsep{3em}
  \xymatrix{ 
    R &\ar_{p}[l] R &\ar_{0}[l] R^2 &\ar_{\left(\begin{smallmatrix}p & 2 \\ 0 & p\end{smallmatrix}\right)}[l] R^2&\ar_{0}[l] R^2 &\ar_{\left(\begin{smallmatrix}p & 2 \\ 0 & p\end{smallmatrix}\right)}[l] \cdots,
  }
  \]
  and if $R$ has no $2$-torsion, the Hochschild homology is 
  \[
    \HH_i(\Rpp)=\begin{cases}R/p & i=0 \\ R/p^2 & i>0 \textup{ even} \\ 0 & \textup{else.}\end{cases}
  \]
  \end{corollary}
  \begin{proof}
    The total complex is immediate. For the Hochschild homology, note that if there is no $2$-torsion and $p$ is odd, the Smith normal form of $\left(\begin{smallmatrix}p & 2 \\ 0 & p\end{smallmatrix}\right)$ is calculated as
    \[
      \begin{pmatrix}p & 2 \\ 0 & p\end{pmatrix} \Rightarrow \begin{pmatrix}1 & 2 \\ -\lfloor\frac{p}{2}\rfloor p & p\end{pmatrix} \Rightarrow \begin{pmatrix}1 & 0 \\ 0 & 2\lfloor\frac{p}{2}\rfloor p + p\end{pmatrix} = \begin{pmatrix}1 & 0 \\ 0 & p^2\end{pmatrix},
    \]
    Where $\lfloor\frac{p}{2}\rfloor=\frac{p-1}{2}$ as $p$ is odd.

    After dropping the summands that are matched by identities, the total complex becomes
    \[
    \xymatrixrowsep{1.5em}
    \xymatrixcolsep{3em}
    \xymatrix{ 
      R &\ar_{p}[l] R &\ar_{0}[l] R &\ar_{p^2}[l] R&\ar_{0}[l] R &\ar_{p^2}[l] \cdots,
    }
    \]
    so without $p$-torsion, the result follows. (If $p=2$, then the differential instead becomes $\left(\begin{smallmatrix}p & 0 \\ 0 & p\end{smallmatrix}\right)$, and $R/p^2$ is replaced by $R/p\oplus R/p$.)
  \end{proof}

For the cyclic complex, there is again a simplified description based on the simplified Hochschild complex \cite[5.3.9]{Lod}: For a dg-$R$-algebra $A$ with underlying graded tensor algebra $T_R(V)$, the Connes operator can be identified as
\begin{align*}
  \tilde B=\begin{pmatrix}0 & 0 \\\gamma & 0\end{pmatrix}:A_n\oplus (A\otimes V)_{n-1} & \longrightarrow A_{n+1}\oplus (A\otimes V)_n 
\end{align*}
with $\gamma(a)=\varphi(1\otimes a)$, where $\varphi:A\otimes A\to A\otimes V$ is the map from above. 

Untangeling definitions for our setting, we obtain
\begin{align*}
  \gamma: (\Rpp)_n=\langle x^n\rangle & \longrightarrow \langle x^{n-1}\otimes x\rangle= (\Rpp\otimes \langle x\rangle)_n\\
  x^n & \longmapsto \sum_{i=1}^n (-1)^{(n-i)i} x^{n-1}\otimes x = \begin{cases}nx^{n-1}\otimes x & n \textup{ odd}\\ 0 & n \textup{ even.}\end{cases}
\end{align*}

This lets us describe the cyclic bicomplex for $\Rpp$:
\begin{proposition} \label{prop: bicomplex}
  Let $p$ be odd, $R$ without $p$-torsion and $\frac{1}{2}\in R$. Then the cyclic bicomplex for the $R$-algebra $\Rpp$ is quasi-isomorphic to
  \[
  \xymatrixrowsep{1.5em}
  \xymatrixcolsep{2em}
  \xymatrix{
   \vdots  \ar_{p^2}[d] &  \vdots  \ar_{0}[d] &  \vdots  \ar_{p^2}[d] &  \vdots  \ar_{0}[d] &  \vdots  \ar_{p}[d] &  \\
  R \ar_{0}[d]& R \ar_{3}[l] \ar_{p^2}[d] & R\ar_{0}[l]\ar_{0}[d]& R\ar_{1}[l]\ar_{p}[d] & R\ar_{0}[l]  \ar@{.}[ur] \\
  R \ar_{p^2}[d]&R\ar_{0}[l]\ar_{0}[d] & \ar_{1}[l] R \ar_{p}[d] & R\ar_{0}[l] \\
  R\ar_{0}[d]&R\ar_{1}[l]\ar_{p}[d]&R\ar_{0}[l] \\
  R\ar_{p}[d] & R \ar_{0}[l] \\
  R
  }
  \]
  where the bottom $R$ is in bidegree $(0,0)$. Likewise, one obtains the periodic bicomplex, by continuing this to the left, and the negative bicomplex, by dropping the positive-degree columns from the periodic complex.
\end{proposition}
\begin{proof}
  By plugging the previous calculations into Loday's bicomplex, we obtain
  \[
  \xymatrixrowsep{1.5em}
  \xymatrixcolsep{2em}
  \xymatrix{
   \vdots  \ar_{\pi}[d] &  \vdots  \ar_{0}[d] &  \vdots  \ar_{\pi}[d] &  \vdots  \ar_{0}[d] &  \vdots  \ar_{p}[d] &\\
  R^2 \ar_{0}[d]& R^2 \ar_{B_3}[l] \ar_{\pi}[d] & R^2\ar_{0}[l]\ar_{0}[d]& R\ar_{B_1}[l]\ar_{p}[d] & R\ar_{0}[l]  \ar@{.}[ur] \\
  R^2 \ar_{\pi}[d]&R^2\ar_{0}[l]\ar_{0}[d] & \ar_{B_1}[l] R \ar_{p}[d] & R\ar_{0}[l] \\
  R^2\ar_{0}[d]&R\ar_{B_1}[l]\ar_{p}[d]&R\ar_{0}[l] \\
  R\ar_{p}[d] & R \ar_{0}[l]\\
  R,
  }
  \]
  where $\pi=\left(\begin{smallmatrix}p & 2 \\ 0 & p\end{smallmatrix}\right)$ and $B_k=\left(\begin{smallmatrix}0&0 \\ k & 0\end{smallmatrix}\right)$ for odd $k>1$, and $B_1=\left(\begin{smallmatrix}0\\1\end{smallmatrix}\right)$.

  From this, we can eliminate an acyclic subcomplex: Consider one of the $R^2$-entries in an odd total degree. This has generators $a_n=x^n$ and $b_n=x^{n-2}\otimes x$. Fix now $n$ odd, and define a new basis in degree $n$ as
  \[
     \alpha_n=\frac{p}{2}b_n-a_n,\hspace{2em} \beta_n=\frac{1}{2}b_n.
  \]
  In the even degree $n-1$, we define
  \[
    \alpha_{n-1}=a_{n-1}+\frac{p}{2}b_{n-1},\hspace{2em} \beta_{n-1}=\frac{1}{2}b_{n-1}.
  \]
  Then the vertical differential becomes
  \begin{align*}
    \pi(\alpha_n)&=\frac{p}{2}\pi(b_n)-\pi(a_n)=\frac{p}{2}(2a_{n-1}+pb_{n-1})-pa_{n-1}=p^2\beta_{n-1},\\
    \pi(\beta_n)&=\frac{1}{2}(2a_{n-1}+pb_{n-1})= \alpha_{n-1}.
  \end{align*}
  The horizontal differential is
  \begin{align*}
    B_n(\alpha_n)&=\frac{p}{2}B_n(b_n)-B_n(a_n)=-nb_{n-1}=-2n\beta_{n-1},\\
    B_n(\beta_n)&=0.
  \end{align*}
  In even total degree, all differentials vanish. Therefore all odd-degree $\beta_n$ and even-degree $\alpha_{n-1}$ split off as an acyclic subcomplex, and the remaining generators are mapped as described in the proposition, up to a multiplication by the unit $-2$.
\end{proof}

This allows for a more explicit description of the cyclic homology, periodic homology, and negative cyclic homology of $\Rpp$ over $R$: 
Note first that for all three bicomplexes, the total differentials from even to odd total degrees are $0$. If $R$ has no $n$-torsion for all positive $n$, the total differentials from odd to even degrees are injective. Hence we only need to compute the cokernels of the latter differentials. Explicitly, we have the following:

\begin{proposition}\label{cyclic}
  Let $p$ be odd, $R$ without $n$-torsion  for all positive $n$, and $\frac{1}{2}\in R$. Then $\HC_0(\Rpp)=R/p$, $\HC_i(\Rpp)=0$ for odd $i$, and for even positive $i$, $\HC_i(\Rpp)$ is the cokernel of the map
  \[ R^{\frac{i}{2}+1} \to R^{\frac{i}{2}+1} \]
  sending $(x_1, x_2, \dots, x_{\frac{i}{2}+1})$ to $(px_1, x_1+p^2x_2, 3x_2+p^2x_3, \dots, (i-1)x_{\frac{i}{2}} +p^2x_{\frac{i}{2}+1})$.

  This can be identified with the colimit of the following diagram of $R$-modules: 
  \[
\xymatrixrowsep{1.5em}
\xymatrixcolsep{2em}
\xymatrix{
0 & \ar[l] R \ar_{p^2}[d] \\
&R & R \ar_{i-1}[l] \ar_{p^2}[d] \\
&& R & \ar_{i-3}[l] R  \ar@{.}[dr] \\
&&&&R & R \ar_{5}[l] \ar_{p^2}[d] \\
&&&&&R & R \ar_{3}[l] \ar_{p^2}[d] \\
&&&&&&R & \ar_{1}[l] R \ar_{p}[d] \\
&&&&&&&R
}
\]
\end{proposition}

In the periodic situation, this diagram just continues to the left. However, there is a subtlety: Periodic and negative cyclic homology are computed by taking the product totalization of the respective bicomplexes \cite[5.1.2]{Lod}.

\begin{proposition}\label{periodic}
  Let $p$ be odd, $R$ without $n$-torsion  for all positive $n$, and $\frac{1}{2}\in R$. Then $\HP_i(\Rpp)=0$ for odd $i$, and for even $i$, $\HP_i(\Rpp)$ is the cokernel of the map
\[\prod_{\mathbb{N}} R \to \prod_{\mathbb{N}} R \]
sending $(x_1, x_2, \dots)$ to $(px_1, x_1+p^2x_2, 3x_2+p^2x_3, 5x_3+p^2x_4, \dots)$. 

  This can be identified with a completion of the colimit of the following diagram of $R$-modules (see \cite[5.1.9]{Lod}): 
  \[
\xymatrixrowsep{1.5em}
\xymatrixcolsep{2em}
\xymatrix{
 \ar@{.}[dr]  &\\
&R & R \ar_{5}[l] \ar_{p^2}[d] \\
&&R & R \ar_{3}[l] \ar_{p^2}[d] \\
&&&R & \ar_{1}[l] R \ar_{p}[d] \\
&&&&R
}
\]
\end{proposition}

For negative cyclic homology, the periodic diagram is truncated at the other end:
\begin{proposition}\label{negative}
  Let $p$ be odd, $R$ without $n$-torsion  for all positive $n$, and $\frac{1}{2}\in R$. Then $\HC^-_i(\Rpp)=\HP_0(\Rpp)$, for $i \leq 0$, and $\HC^-_i(\Rpp)=0$ for odd $i$, and for even positive $i$, $\HC^-_i(\Rpp)$ is the cokernel of the map
\[\prod_{\mathbb{N}} R \to \prod_{\mathbb{N}} R \]
sending $(x_1, x_2, \dots)$ to $(p^2x_1, (i+1)x_1+p^2x_2, (i+3)x_2+p^2x_3, (i+5)x_3+p^2x_4, \dots)$. 

   This can be identified with a completion of the colimit of the following diagram of $R$-modules (see \cite[5.1.9]{Lod}): 
   \[
\xymatrixrowsep{1.5em}
\xymatrixcolsep{2em}
\xymatrix{
 \ar@{.}[dr]  \\
&R & R \ar_{i+5}[l] \ar_{p^2}[d] \\
&&R & R \ar_{i+3}[l] \ar_{p^2}[d] \\
&&&R & \ar_{i+1}[l] R \ar_{p^2}[d] \\
&&&&R
}
\]
\end{proposition}

The rest of this paper is entirely devoted to identifying these cokernels.

For readability, we will set up some notation for the staircase diagrams and homology groups:

Let $N_{-1}=\HC_0(\Rpp)=R/p$ and for each odd integer $i \geq 1$, and $N_i=\HC_{i+1}(\Rpp)$. By \ref{cyclic}, $N_i$ is the colimit of the following diagram of $R$-modules:
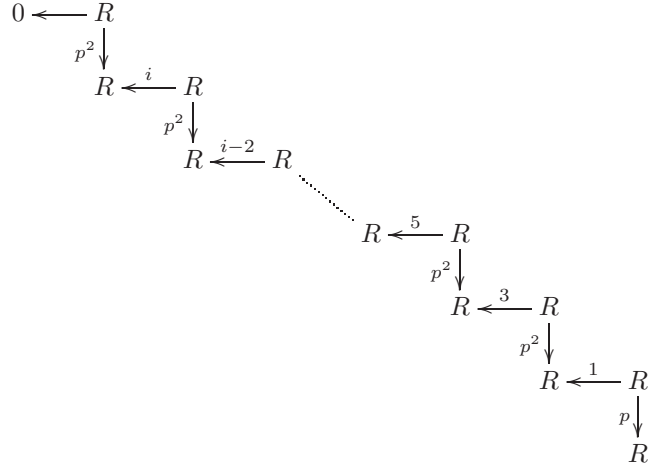
\[
\xymatrixrowsep{1.5em}
\xymatrixcolsep{2em}
\xymatrix{
0 & \ar[l] R \ar_{p^2}[d] \\
&R & R \ar_{i}[l] \ar_{p^2}[d] \\
&& R & \ar_{i-2}[l] R  \ar@{.}[dr]  \\
&&&&R & R \ar_{5}[l] \ar_{p^2}[d] \\
&&&&&R & R \ar_{3}[l] \ar_{p^2}[d] \\
&&&&&&R & \ar_{1}[l] R \ar_{p}[d] \\
&&&&&&&R
}
\]
So that we can refer to the individual copies of $R$ in this diagram more easily, we will use the following labels:
\begin{equation}
\label{staircase-head}
\begin{split}
\xymatrixrowsep{1.5em}
\xymatrixcolsep{2em}
\xymatrix{
0 & \ar[l] R_{i+2}' \ar_{p^2}[d] \\
&R_i & R_i' \ar_{i}[l] \ar_{p^2}[d] \\
&& R_{i-2} & \ar_{i-2}[l] R_{i-2}'  \ar@{.}[dr]  \\
&&&&R_5 & R_5' \ar_{5}[l] \ar_{p^2}[d] \\
&&&&&R_3 & R_3' \ar_{3}[l] \ar_{p^2}[d] \\
&&&&&&R_1 & \ar_{1}[l] R_1' \ar_{p}[d] \\
&&&&&&&R_{-1}
}
\end{split}
\end{equation}
All of the above terms $R_1', R_3', \ldots, R_{-1}, R_1, R_3, \ldots$ are equal to $R$.

We note the natural maps between the $R$-modules $N_i$ and record some of their properties.

\begin{lemma} \label{between heads}
Let $k \geq i \geq 1$ be odd integers, and let $N_k$ and $N_i$ be the colimits defined above.  There is a natural surjective $R$-module homomorphism $\pi_{k,i}: N_k \rightarrow N_i$, and these are compatible in the sense that for $k' \geq k \geq i$, we have $\pi_{k,i} \circ \pi_{k',k} = \pi_{k',i}$.  The kernel of $\pi_{k,i}$ is the $R$-submodule of $N_k$ generated by the images of the natural maps $R_j \rightarrow N_k$ for all $j > i$.
\end{lemma}

\begin{proof}
There is a clear map $N_{k} \rightarrow N_i$ given by mapping the $R_{j}$ and $R'_{j+2}$ factors to zero for all $j > i$, and by mapping the remaining factors to themselves via the identity. Using the explicit description of an (unfiltered) colimit given, for example, in \cite[\href{http://stacks.math.columbia.edu/tag/08QQ}{Tag~00D5}]{stacks-project}, we have that $N_i$ is a quotient of the direct sum $R_{-1} \oplus R_1 \oplus \cdots \oplus R_k \oplus \cdots \oplus R_{i}$, and similarly for $N_k$.  The constructed map $N_k \rightarrow N_i$ is induced by the obvious projection between these direct sums.  The rest of the claims follow easily.
\end{proof}

We let $N_{\infty}$ denote the inverse limit $\varprojlim N_i$.  We let $K_i$ denote the kernel of the projection $N_{\infty} \rightarrow N_i$.  It is immediate from \ref{periodic} and \ref{negative} that $N_\infty=\HP_0(\Rpp)$ and $K_i=\HC_{i+3}^-(\Rpp)$ (see \cite[5.1.5 and 5.1.9]{Lod}). Indeed, this follows since the cyclic homology of $\Rpp$ is even (Proposition \ref{cyclic}) and hence by \cite[5.1.5]{Lod}, we have a short exact sequence
\[0 \to \HC_{l}^-(\Rpp) \to HP_l(\Rpp) \to \HC_{l-2}(\Rpp) \to 0\]
for any even integer $l$. Our goal is therefore to compute these terms $N_i, N_{\infty}, K_i$ as explicitly as possible.

Our explicit description of $N_i$ is quite complicated for general $i$.  (We will see that for certain $i$, such as $i = \frac{p^a \pm 1}{2}$, there is a simple description.)  Our strategy for calculating $N_i$ explicitly is to first construct a related $R$-module, that we call $M_i$, and that has a much more regular description in general than $N_i$.

\section{Computation of a related colimit}  From now on we will work with $p$-torsion-free commutative $\ZZ_{(p)}$-algebras. This will make our formulas and results easier to formulate. There is no restriction of generality since for any $p$-torsion-free $R$, the map $R \to R_{(p)}$ induces an isomorphism $\xymatrix{\HH^R(\Rpp) \ar[r]^-{\cong} & \HH^{R_{(p)}}(\Rppp)}$ by Corollary \ref{cor: Hochschild}. Proposition \ref{prop: bicomplex} then implies that corresponding relative negative, periodic and cyclic homologies are isomorphic. For the rest of the paper $R$ will denote a $p$-torsion-free commutative $\ZZ_{(p)}$-algebra. 

For every odd integer $i \geq 1$, we will prove that the $R$-module $M_i := R \oplus R/1 \oplus R/3 \oplus \cdots \oplus R/i$ is a colimit of the diagram 
\begin{equation}
\label{staircase-regular}
\begin{split}
\xymatrixrowsep{1.5em}
\xymatrixcolsep{2em}
\xymatrix{
R_i & R_i' \ar_{i}[l] \ar_{p^2}[d] \\
& R_{i-2} & \ar_{i-2}[l] R_{i-2}'  \ar@{.}[dr]  \\
&&&R_5 & R_5' \ar_{5}[l] \ar_{p^2}[d] \\
&&&&R_3 & R_3' \ar_{3}[l] \ar_{p^2}[d] \\
&&&&&R_1 & \ar_{1}[l] R_1' \ar_{p}[d] \\
&&&&&&R_{-1}
}
\end{split}
\end{equation}

Here, as in diagram~(\ref{staircase-head}), all these terms $R_i$ and $R_i'$ are equal to $R$.  We use the subscripts so we can refer to specific terms more easily. 
Note that this is the same as the diagram~(\ref{staircase-head}) defining $N_i$, except that the top row of the $N_i$ diagram has been removed.  In particular, we have a surjective map $M_i \twoheadrightarrow N_i$ for every odd $i$.

For the proof that $M_i$ is a colimit of diagram~(\ref{staircase-regular}), we first recursively define two sequences $a_j$ and $b_j$, and prove some inequalities related to them.  Some of these results will be needed immediately, in the proof that $M_i$ is a colimit of diagram~(\ref{staircase-regular}), while other of these results will not be needed until later.  These sequences will also be used to describe the kernel of the surjection $M_i \rightarrow N_i$.

\begin{definition} \label{abdef}
Define $A_1 := p$ and for each odd integer $j \geq 3$, recursively define $A_j \in \QQ$ as $A_j := \frac{p^2 A_{j-2}}{j}$.  Define $B_0 := 1$ and for each even integer $j \geq 2$, define $B_j \in \QQ$ as $B_j := \frac{p^2 B_{j-2}}{j}$.  Define $a_j := v_p(A_j)$ and $b_j := v_p(B_j)$.
\end{definition}

We will see below in Corollary~\ref{AB in Zp} that the elements $A_j$ and $B_j$ defined above in fact lie in $\ZZ_{(p)}$

Note that the sequence $b_j$ is not an increasing sequence.  We can bound the value of $b_j$ as follows.

\begin{lemma} \label{b bounds}
For every even integer $j \geq 0$, the number $b_j$ is equal to $j$ minus the $p$-adic valuation of $(j/2)!$.  In particular, for every even integer $j \geq 2$, we have 
\[
j - \frac{j}{2(p-1)} < b_j.
\]
\end{lemma}

\begin{proof}
The statement about $(j/2)!$ is clear from the recursive definition of the $b$ sequence.  From this and Legendre's formula, we have for all $j \geq 0$ that
\begin{align*}
b_j &=j-\nu_p((j/2)!) =j - \sum_{k = 1}^{\infty} \left\lfloor \frac{j}{2p^k} \right\rfloor. \\
\intertext{So if $j > 0$, we have}
b_j&> j - \sum_{k = 1}^{\infty} \frac{j}{2p^k} = j - \frac{j}{2(p-1)}.
\end{align*}
This proves the inequality.
\end{proof}

The following coarse lower-bound on $b_j$ will be very useful.

\begin{lemma} \label{padic val}
Assume $p^e \leq \frac{j}{2} + 1$, for an integer $e \geq 0$ and an even integer $j \geq 0$.  Then $b_j \geq e$.
\end{lemma}

\begin{proof}
The claim holds if $j = 0$ and $j=2$, and for $j > 2$, we have $b_j > j - \frac{j}{2(p-1)} \geq \frac{j}{2}+1 \geq p^e \geq e$.
\end{proof}

\begin{lemma}
For every odd integer $j \geq 1$, we have $a_j = b_{2j} - b_{j-1} - 1$.
\end{lemma}

\begin{proof}
The claim holds when $j = 1$.  Now, assuming the result for some fixed value of $j - 2$, we compute (using that $p \neq 2$ several times)
\begin{align*}
a_j &= 2 + a_{j-2} - v_p(j) \\
&= 2 + \big(b_{2(j-2)} - b_{j-3} - 1\big) - v_p(j) \\
&= 2 + b_{2(j-2)} - v_p(2(j-1)) + v_p(j-1) - b_{j-3} - 1 - v_p(j) \\
&= b_{2j-2} - b_{j-1} + 1  - v_p(j) \\
&= b_{2j} - b_{j-1} - 1, 
\end{align*}
which completes the induction.
\end{proof}

\begin{corollary} \label{AB in Zp}
We have $b_j \geq 0$ for every even integer $j \geq 0$.  In particular, the element $B_j \in \QQ$ in fact lies in $\ZZ_{(p)}$, and so can be viewed as an element of $R$.  Similarly, $A_j \in \ZZ_{(p)}$ for every odd integer $j \geq 1$.
\end{corollary}

\begin{proof}
The statement about $b_j$ is clear from the above results.  For the numbers $a_j$, the result holds for $j = 1$, and for all values of $j \geq 3$, we can use the bounds 
\[
a_j = b_{2j} - b_{j-1} - 1 \geq 2j - \frac{2j}{2(p-1)} - (j-1) - 1 = j - \frac{j}{p-1}.
\]
This completes the proof.
\end{proof}

Our next preliminary result will provide a way to determine the colimit of the diagram~({\ref{staircase-regular}}) for a fixed value of $i$ in terms of the colimit for $i-2$.

\begin{lemma} \label{divide by k lemma}
Fix an integer $k$.  Assume $\psi: R \rightarrow M$ is an $R$-module homomorphism and that there exists $m_0 \in M$ such that $\psi(1) = km_0$. Let $\iota: M \rightarrow M \oplus R/k$ be given by $m \mapsto (m,0)$ and let $\psi': R \rightarrow M \oplus R/k$ be given by $r \mapsto (rm_0, r)$.  Then 
\[
\xymatrix{
R \ar^{\psi'}[d] & R \ar_{k}[l] \ar^{\psi}[d] \\
M \oplus R/k& \ar_-{\iota}[l]M 
}
\]
is a pushout diagram. 
\end{lemma}

\begin{proof}
It is clear that the given diagram commutes.  
Assume we have a commutative diagram
\[
\xymatrix{
&R \ar@/_1.5pc/ _{f_1}[ddl] \ar^{\psi'}[d] & R \ar_{k}[l] \ar^{\psi}[d] \\
& M \oplus R/k& \ar_-{\iota}[l]M \ar@/^1.5pc/^{f_2}[dll]\\
N
}
\]
Define a map $f:M \oplus R/k \rightarrow N$ by $(m,r) \mapsto f_1(r) + f_2(m-rm_0)$.  This map is well-defined because we have
\[
f_1(r+sk) + f_2(m-rm_0 - skm_0) = f_1(r) + f_2(m - rm_0) + kf_1(s) - f_2 (\psi(s)) = f_1(r) + f_2(m - rm_0).
\]
This map is an $R$-module homomorphism. 

It follows directly from the definitions that $f_1 = f \circ \psi'$ and $f_2 = f \circ \iota$. Hence the diagram
\[
\xymatrix{
&R \ar@/_1.5pc/_{f_1}[ddl] \ar^{\psi'}[d] & R \ar_{k}[l] \ar^{\psi}[d] \\
& M \oplus R/k \ar_{f}[dl] & \ar_-{\iota}[l]M \ar@/^1.5pc/^{f_2}[dll]\\
N
}
\]
commutes.
 
Lastly we check that our map $f: M \oplus R/k \rightarrow N$ is the unique map making the diagram commute.  Let $g$ be any map making the diagram commute.  We then must have $g(m,0) = f_2(m)$ for all $m \in M$, and we must have $g(rm_0, r) = f_1(r)$ for all $r \in R$.  Thus we must have $g(m,r) = g(rm_0, r) + g(m - rm_0, 0) = f_1(r) + f_2(m - r m_0) = f(m,r)$ for all $m \in M, r \in R$, as required.
\end{proof}

We now use the above preliminary results to compute the colimit of the diagram~(\ref{staircase-regular}).

\begin{proposition} \label{colimit-regular}
For every odd positive integer $i$, set $M_i := R \oplus R/1 \oplus R/3 \oplus \cdots \oplus R/i$.  Then $M_i$ is a colimit of the diagram~(\ref{staircase-regular}). 
\end{proposition}

\begin{proof}
We prove this using induction on $i$.  For the base case $i=1$, consider the diagram of $R$-modules
\[
\xymatrix{
R \ar^{\phi_1}[d] & R \ar_{1}[l] \ar^{p}[d] \\
R \oplus R/1& \ar_-{\iota}[l] R,
}
\]
with $\phi_1(1) = (p,0)$ and $\iota(1) = (1,0)$.  In this case we have that $M_1$, with the indicated maps, is a colimit of the diagram~(\ref{staircase-regular}).  For later use in the induction, we note also that $\phi_1(1) \in R \oplus R/1$ can be represented by $(A_1, B_0)$, as defined in Definition~\ref{abdef}.

Now assume the result has been proven for some fixed value of $i-2$, and let $\phi_{i-2}$ denote the corresponding map $R_{i-2} \rightarrow M_{i-2}$, where the notation $R_{i-2} = R$ refers to the upper-leftmost factor in diagram~(\ref{staircase-regular}). Further, assume that $\phi_{i-2}(1)$ can be represented by the element $(A_{i-2}, B_{i-2-1}, B_{i-2-3}, \ldots, B_0)$. 

Consider the diagram of $R$-modules
\[
\xymatrix{
R \ar^-{\phi_i}[d] & R \ar_{i}[l] \ar^-{p^2 \phi_{i-2}}[d] \\
 M_{i-2} \oplus R/i & \ar_-{\iota}[l] M_{i-2},
}
\]
where $\iota(m) = (m,0)$ and where $\phi_i(1) = (A_{i}, B_{i-1}, B_{i-3}, \ldots, B_{2}, B_0)$.  By Lemma~\ref{divide by k lemma}, to prove that $M_{i-2} \oplus R/i$ (with the above maps) is a colimit of the diagram
\[
\xymatrix{
R  & R \ar_{i}[l] \ar^-{p^2 \phi_{i-2}}[d] \\
& M_{i-2},
}
\]
it suffices to check that $iA_i = p^2 A_{i-2} \in R$ and that for every odd integer $n$ in the interval $1 \leq n \leq i-2$, we have $iB_{i-n} = p^2 B_{i-2-n} \in R/n$.  The first equality is clear from the definition of $A_i$.  For the second equality, it is clear again from the definition that $(i-n)B_{i-n} = p^2 B_{i-2-n}$, so the desired result follows, since $nB_{i-n} = 0 \in R/n$.
\end{proof}

We record an important aspect of the proof of Proposition~\ref{colimit-regular} in the following corollary.

\begin{corollary} \label{pushout valuations}
Fix odd integers $i \geq j \geq 1$.  Because $R \oplus R/1 \oplus R/3 \oplus \cdots \oplus R/i$ is a colimit of the diagram~(\ref{staircase-regular}), we have a corresponding map of $R$-modules $\phi_{j,i}: R_j \rightarrow R \oplus R/1 \oplus R/3 \oplus \cdots \oplus R/i$.  This map has the property that $\phi_{j,i}(1)$ can be represented by an element $(c_{j,-1}, c_{j,1}, \ldots, c_{j,i})$, where each $c_{j,n} \in \ZZ_{(p)}$, and where these elements satisfy the following properties (with the $A$ and $B$ sequences defined as in Definition~\ref{abdef}):
\begin{enumerate}
\item We have $c_{j,-1} = A_j$.
\item If $n$ is an odd integer in the range $j < n \leq i$, then $c_{j,n} = 0$.
\item If $n$ is an odd integer in the range $1 \leq n \leq j$, then $c_{j,n} = B_{j-n} \in R/n$.
\end{enumerate}
\end{corollary}

\begin{proof}
This follows directly from the construction given in the proof of Proposition~\ref{colimit-regular}.
\end{proof}

\begin{corollary} \label{comp Mi to Ni}
Fix an odd positive integer $i$, and let $\phi_{i,i}: R_i \rightarrow M_i$ be as in Corollary~\ref{pushout valuations}.
The $R$-module $N_i$ (as defined in Section~\ref{setup section}) is isomorphic to the cokernel of $p^2 \phi_{i,i}: R_i \rightarrow M_i$.
\end{corollary}

\begin{proof}
This follows immediately from comparing the diagrams used to define $M_i$ and $N_i$.
\end{proof}

\section{Computation of $N_i$ for certain $i$, and the computation of $N_{\infty}$} 

Recall again that $R$ denotes a $p$-torsion-free commutative $\ZZ_{(p)}$-algebra. Our goal in this section is to give an explicit description of $N_i$ for certain values of $i$.  Namely, we define an infinite set $Z_1$ in Definition~\ref{def of Z1} below, and compute $N_i$ for all $i \in Z_1$.  This will also enable us to compute the inverse limit $\varprojlim N_i =: N_{\infty}$.

By Corollary~\ref{comp Mi to Ni}, to compute $N_i$, it is equivalent to compute the cokernel of the map $p^2 \phi_{i,i}: R_i \rightarrow M_i$, where $\phi_{i,i}$ is defined as in Corollary~\ref{pushout valuations}.  For certain values of $i$, this map $\phi_{i,i}$ turns out to have a particularly simple form.  The description of those values of $i$ will involve the following parameter $g(n)$.

\begin{definition} \label{def gn}
Let $n \geq 1$ denote an integer which is divisible by $p$.  We define $g(n)$ to be the maximum even integer $j$ such that $b_j < v_p(n)$.
\end{definition}

The motivation for the definition of $g(n)$ is the following.  We will show that the map $\phi_{i,i}: R_i \rightarrow R \oplus R/1 \oplus R/3 \oplus \cdots \oplus R/i$ is non-zero in the $R/n$-factor only if $i$ is in the interval $[n, n+g(n)]$.

\begin{example} \label{ex gn}
We have $g(p) = g(p^2) = 0$, $g(p^3) = g(p^4) = 2$, $g(p^5) = 4$, and $g(3^6) = 6$. For example, if $\phi_{i,i}$ is non-zero in the $R/p^4$-factor, we will see below that the only possible values for $i$ are $i = p^4$ and $i = p^4 + 2$.
\end{example}

Let $c_{i,n}$ be defined as in Corollary~\ref{pushout valuations}, and let $b_j$ be defined as in Definition~\ref{abdef}.  The following result follows immediately from Corollary~\ref{pushout valuations} and the definitions.

\begin{lemma} \label{cin 0}
Let $n$ be an odd positive integer.  Assume $c_{j,n} \neq 0 \in R/n$ for some $j > n$ (so in particular, $v_p(n) \geq 1$).  Then $j - n \leq g(n)$.
\end{lemma}

\begin{proof}
We know $v_p(c_{j,n}) = b_{j-n}$ by Corollary~\ref{pushout valuations}.  Since we are assuming that $c_{j,n} \neq 0 \in R/n$, we must have $b_{j-n} < v_p(n)$.  So $g(n) \geq j - n$ by the definition of $g(n)$, as claimed.
\end{proof}

\begin{definition} \label{def of Z1}
Let $X$ denote the set of all odd positive integers, and let $X_1$ denote the set of all odd positive integers which are divisible by $p$. Define $Z_1$ to be the set 
\[
Z_1 := X \setminus \bigcup_{n \in X_1} [n, n+g(n)].
\]
\end{definition}

For all $i \in Z_1$, the map $\phi_{i,i}: R \rightarrow M_i$ has a particularly simple form.

\begin{lemma} \label{some Ni}
Assume $i \in Z_1$.  Then the map $\phi_{i,i} : R_i \rightarrow M_i = R \oplus R/1 \oplus R/3 \oplus \cdots \oplus R/i$ has the property that $\phi_{i,i}(1) = (A_{i}, 0, 0, \ldots, 0),$ where $A_i$ is defined as in Definition~\ref{abdef}.  In particular, for each such $i$ we have an isomorphism of $R$-modules
\[
N_i \cong R/p^{a_{i}+2} \oplus R/1 \oplus R/3 \oplus \cdots \oplus R/(i-2) \oplus R/i.
\]
\end{lemma}

\begin{proof}
This follows immediately from Lemma~\ref{cin 0} and the definitions.
\end{proof}

\begin{corollary} \label{projection Z1}
Assume $i \in Z_1$ and let $k > i$ be an odd integer.  Then the projection map $N_k \rightarrow N_i$ from Lemma~\ref{between heads} is induced by a map $M_k \rightarrow N_i$,
\begin{align*}
R \oplus R/1 \oplus R/3 \oplus \cdots \oplus R/i \oplus \cdots \oplus R/k &\rightarrow R/p^{a_i+2} \oplus R/1 \oplus R/3 \oplus \cdots \oplus R/i,\\
\intertext{of the form}
(r_{-1},r_1,r_3,\ldots, r_{i}, \cdots, r_k) &\mapsto (r_{-1}+C,r_1,r_3,\ldots,r_i),
\end{align*}
where $C$ is divisible by $p^a$, for $a = \min_{m \geq i+2} a_m$ and and it only depends on $(r_{i+2}, \dots, r_k)$. 
\end{corollary}

\begin{proof}
Using the construction and notation from the proof of Lemma~\ref{between heads}, the projection map $N_k \rightarrow N_i$ fits into a commutative diagram
\[
\xymatrix{
R_{-1} \oplus R_1 \oplus \cdots \oplus R_i \oplus \cdots \oplus R_k \ar[r] \ar[d]_{\phi} & R_{-1} \oplus R_1 \oplus \cdots \oplus R_i \ar[dd] \\
M_k = R \oplus R/1 \oplus R/3 \oplus \cdots \oplus R/i \oplus \cdots \oplus R/k \ar[d] \\
N_k \ar[r] & N_i,
}
\]
where all of the $R_j$ terms are equal to $R$.  In this diagram, the top horizontal map is the obvious projection.

Consider first an element of the form
\[
(r_{-1},r_1,r_3,\ldots, r_{i}, 0, \cdots, 0) \in R \oplus R/1 \oplus R/3 \oplus \cdots \oplus R/i \oplus \cdots \oplus R/k \cong M_k.
\]
Choose elements $s_{-1}, s_1, \ldots, s_i \in R$ so that the element
\[
(s_{-1}, s_1, \ldots, s_i, 0, \ldots, 0) \in R_{-1} \oplus R_1 \oplus \cdots \oplus R_i \oplus \cdots \oplus R_k
\]
maps under $\phi$ to $(r_{-1},r_1,r_3,\ldots, r_{i}, 0, \cdots, 0) \in M_k$. This is possible from the construction of $\phi$ (Corollary \ref{pushout valuations}), since we can choose
\[
(s_{-1}, s_1, \ldots, s_i) \in R_{-1} \oplus R_1 \oplus \cdots \oplus R_i
\]
mapping to $(r_{-1},r_1,r_3,\ldots, r_{i}) \in M_i$, and hence also to an element with the same representative in $N_i$.  Thus our claimed result is true for elements of the form $(r_{-1},r_1,r_3,\ldots, r_{i}, 0, \cdots, 0) \in M_k$.  (In fact, for these special elements, a stronger result is true, because we have showed that in this case, we have $C = 0$.)

Because the map $M_k \rightarrow N_i$ is an $R$-module map, it in particular is additive, so it suffices now to consider the ``complementary" elements, i.e., those of the form
\[
(0,\ldots, 0, r_{i+2}, \cdots, r_k) \in R \oplus R/1 \oplus R/3 \oplus \cdots \oplus R/i \oplus \cdots \oplus R/k \cong M_k.
\]
We will show that this element is the image under $\phi$ of an element 
\[
(C, 0, \ldots, 0, s_{i+2}, \ldots, s_k) \in R_{-1} \oplus R_1 \oplus \cdots \oplus R_i \oplus \cdots R_k,
\]
where $C$ is as in the statement.  Under the top horizontal map, this element $(C, 0, \ldots, 0, s_{i+2}, \ldots, s_k)$ maps to $(C, 0, \ldots, 0)$, and this will complete the proof.

We now carry out the construction just described.  Let $(0,\ldots, 0, r_{i+2}, \cdots, r_k) \in M_k$, and choose any pre-image under $\phi$, denoted $(y_{-1}, y_1, \ldots, y_i, y_{i+2}, \ldots, y_k)$.  Because $i \in Z_1$, by Corollary~\ref{pushout valuations} and Lemma~\ref{cin 0}, we know $c_{t,l}=0$ for $1 \leq l \leq i$ and $i \leq t$. Hence the image of $(0, \ldots, 0, y_{i+2}, \ldots, y_k)$ under $\phi$ must be $(-C, 0,\ldots, 0, r_{i+2}, \cdots, r_k)$, for some $C \in R$ that is divisible by $p^a$, for $a = \min_{m \geq i+2} a_m$, and hence $(0,\ldots, 0, r_{i+2}, \cdots, r_k) \in M_k$ is the image under $\phi$ of $(C, 0, \ldots, 0, y_{i+2}, \ldots, y_k)$.  Under the top horizontal map in our diagram, this element maps to $(C, 0, \ldots, 0)$, which in turn maps to an element in $N_i$ represented by $(C, 0, \ldots, 0)$.  This completes the proof.
\end{proof}

Because of the $C$ terms in the above formula, the maps $N_j \rightarrow N_i$ for $i,j \in Z_1$ typically do not correspond to the obvious projections. (For example, the $R/(i+2)$-component will typically contribute something nonzero under the projection.) This is relevant, as our next goal is to compute the inverse limit $\varprojlim_k N_k$ by restricting to $\varprojlim_{i \in Z_1} N_i$, where these maps turn up. 

To see that $Z_1\subset \mathbb{N}$ is indeed cofinal, we first need the following:

\begin{lemma}\label{estimate gpa}
  For any $a>0$, $g(p^a)<\frac{p^a-1}{2}$.
\end{lemma}
\begin{proof}
  For $a=1$ and $2$, this is immediate since there $g(p^a)=0$.

  From Lemma \ref{b bounds}, we have the linear lower bound 
  \[
    b_j > \left( \frac{2p-3}{2p-2}\right)j.
  \]
  This yields an upper bound on $g(p^a)$:
  \[
    g(p^a)=\max\{j: b_j < a\} \leq \max\left\{j: \left( \frac{2p-3}{2p-2}\right)j < a\right\} < \left( \frac{2p-2}{2p-3}\right)a.
  \]
  After substituting this and rearranging the inequality, it therefore suffices to show
  \[
    p^a - \frac{4p-4}{2p-3}a > 1.
  \]
  Note that  $\frac{4p-4}{2p-3}$ is strictly decreasing in $p$ for $p>2$, and $p^a$ is strictly increasing in $p$ (assuming $a>1$)--hence it suffices to show this for $p=3$. That leaves us with
  \[
    3^a-\frac{8}{3}a>1,
  \]
  which is true by elementary methods. (For example, it holds for $a=2$, and the left hand side is strictly increasing for $a\geq 2$.)
\end{proof}

\begin{corollary}\label{Z-sequence}
  For every odd $n$ that is divisible by $p$, $\lvert\frac{p^a\pm 1}{2} -n\rvert > g(n)$. In particular, all odd numbers of the form $\frac{p^a\pm 1}{2}$ are contained in $Z_1$, which therefore contains arbitrarily large numbers.
\end{corollary}
\begin{proof}
  By the previous lemma, we have
  \[
    g(n)=g( p^{\nu_p(n)})<\frac{p^{\nu_p(n)}-1}{2},
  \]
  therefore it suffices to show that
  \[
    \left|\frac{p^a\pm 1}{2} -n\right| \geq \frac{p^{\nu_p(n)}-1}{2}.
  \]
  In the case that $\nu_p(n)\geq a$, we have
  \[
    n-\frac{p^a\pm 1}{2}\geq p^{\nu_p(n)}-\frac{p^{\nu_p(n)}\pm 1}{2}\geq \frac{p^{\nu_p(n)}- 1}{2}.
  \]
  Otherwise, we write $a=\tilde a + \nu_p(n)$, $n=m p^{\nu_p(n)}$, and observe
  \[
    \left|\frac{p^a\pm 1}{2} -n\right|= \left|\frac{(p^{\tilde{a}}-2m)p^{\nu_p(n)}\pm 1}{2} \right| \geq \left|\frac{p^{\nu_p(n)}\pm 1}{2}\right|\geq \frac{p^{\nu_p(n)} - 1}{2},
  \]
  where the first inequality uses that $p$ is odd, hence $(p^{\tilde{a}}-2m)$ is nonzero.
\end{proof}

%

\begin{proposition} \label{Ninfty}
The inverse limit $N_{\infty} := \varprojlim_k N_k$ is isomorphic to
\[
R^{\wedge} \times R/1 \times R/3 \times R/5 \times \cdots,
\]
where $R^{\wedge}$ denotes the $p$-adic completion of $R$.
\end{proposition}

\begin{proof}
We consider the inverse limit of the final system $N_{i}$ for $i \in Z_1$.  Define a map
\[
\varprojlim_{i \in Z_1} N_{i} \rightarrow R^{\wedge} \times R/1 \times R/3 \times R/5 \times \cdots
\]
as follows.  Given a compatible sequence in $x = (x_i)_{i \in Z_1} \in \varprojlim_{i \in Z_1} N_i$, define its image to be $(y_{-1}, y_1, y_3, \ldots) \in R^{\wedge} \times R/1 \times R/3 \times \cdots$, where the coordinates $y_n$ are defined as follows.  To define $y_n$ with $n \neq -1$, choose any $k \in Z_1$ such that $k \geq n$, and then write $x_k \in N_k$ as $(x_{k,-1}, x_{k,1}, \ldots, x_{k,k})$, and set $y_n := x_{k,n}$.  It follows from Corollary~\ref{projection Z1} that this element $y_n$ is independent of the choice of $k$.

Defining $y_{-1}$ is more difficult, because of the elements denoted by $C$ in Corollary~\ref{projection Z1}.  Fix $i \in Z_1$ and let $a = \min_{m \geq i+2} a_m$.  Then in the notation of the previous paragraph, we have that $x_{k,-1} \equiv x_{i,-1} \bmod p^a$ for all $k \in Z_1$ such that $k \geq i$.  We also have that these values $\min_{m \geq i+2} a_m$ approach infinity as $i$ approaches infinity.  In this way, we form a $p$-adic Cauchy sequence, and we define $y_{-1}$ to be the limit of this sequence in $R^{\wedge}$.

It is clear that the given map $\varprojlim_{i \in Z_1} N_{i} \rightarrow R^{\wedge} \times R/1 \times R/3 \times R/5 \times \cdots$ is an $R$-module homomorphism.  To see that it is injective, consider a sequence $(x_i)$ mapping to $0$.  By the construction, we have that $x_{i,k} = 0$ for all $i$ and all $k \neq -1$.  But now in this case, with all other terms being $0$, the $C$ term from Corollary~\ref{projection Z1} is also equal to $0$, so the terms $x_{i,-1}$ must also equal $0$, for all $i \in Z_1$.

For surjectivity, consider any element $(y_{-1}, y_1, y_3, \ldots) \in R^{\wedge} \times R/1 \times R/3 \times \cdots$.  Note first that some element of the form $(*, y_1, y_3, \ldots) \in R^{\wedge} \times R/1 \times R/3 \times \cdots$ is in the image.  (All components except for the $x_{i,-1}$-components are determined, and the $x_{i,-1}$ components can be calculated inductively, one at a time, for $i \in Z_1$.)   Also, for any element $r' \in R^{\wedge}$, the element $(r', 0, 0, \ldots)$ is in the image.  Surjectivity now follows from the fact that the map is additive. 
\end{proof}

\begin{example}
When $p = 3$, the set $Z_1$ from Definition~\ref{def of Z1} contains the following elements: 5, 7, 11, 13, 17, 19, 23, 25, 31, 35, 37, 41, 43, 47, 49, 53, 55, 59, 61, 65, 67, 71, 73, 77, 79, 85, 89, 91, 95, 97, 101, 103, 107, 109, 113, 115, 119, 121, 125, 127, 131, 133, 139, 143, 145, 149, 151, 155, 157, 161, 163, 167, 169, 173, 175, 179, 181, 185, 187, 193, 197, 199.

In other words, in addition to multiples of $3$, it only excludes the elements $29, 83, 137, 191$ from the odd numbers up to $200$. 
\end{example}

For any odd integers $i \geq j \geq 1$, the map $\phi_{j,i}: R_j \rightarrow M_i$ considered in Corollary~\ref{pushout valuations} induces a map $\psi_{j,i}: R_j \rightarrow N_i$.  When $i \in Z_1$, then by Lemma~\ref{some Ni}, we have $\psi_{j,i}(1) = (c_{j,-1}, c_{j,1}, \ldots, c_{j,i})$, where the values $c_{j,k}$ are as in Corollary~\ref{pushout valuations}, with the only difference being that the initial component $c_{j,{-1}}$ here is considered as an element of $R/p^{a_i+2}$, rather than as an element of $R$.  These maps $\psi_{j,i}$ can also be used to define a map to the inverse limit, $R_j \rightarrow N_{\infty}$, as in the following corollary.

\begin{corollary} \label{psij}
For any odd integer $j \geq 1$, the maps $\psi_{j,i}$, for varying $i \geq j$ with $i \in Z_1$, determine a map $\psi_j: R_j \rightarrow N_{\infty}$, such that $\psi_j(1) = (c_{j,-1}, c_{j,1}, c_{j,3}, \ldots)$.  These values $c_{j,k}$ are the same as described in Corollary~\ref{pushout valuations}.
\end{corollary}

\begin{proof}
This follows from the preceding comments, together with Corollary~\ref{projection Z1} and the construction in the proof of Proposition~\ref{Ninfty}.
\end{proof}

\section{Computation of $K_i$ for certain $i$}

Recall from Section~\ref{setup section} that $K_i$ is defined to be the kernel of the projection map 
\[
N_{\infty} = \varprojlim N_j \rightarrow N_i.
\]
In this section, we compute this kernel for certain values of $i$ in a certain set $Z_2$; this set $Z_2$ is a proper subset of the set $Z_1$ considered above.

\begin{definition} \label{def of Z2}
Let $X$ denote the set of all odd positive integers, and let $X_1$ denote the set of all odd positive integers which are divisible by $p$. Let $g(n)$ be as in Definition~\ref{def gn}.  Define $Z_2$ to be the set 
\[
Z_2 := X \setminus \bigcup_{n \in X_1} [n-g(n), n+g(n)].
\]
\end{definition}

The maps $\psi_j: R_j \rightarrow N_{\infty}$ from Corollary~\ref{psij} can be used to describe the $R$-modules $K_m$, as in the following lemma.  It turns out the value of $m+2$ is more important in the eventual computation than $m$ itself, so that is why we phrase this lemma in terms of $K_{i-2}$. 

\begin{lemma} \label{gens of kernel}
For any odd integer $i>1$, the kernel $K_{i-2}$ of the natural projection $N_{\infty} \rightarrow N_{i-2}$ is the closure (in the inverse limit topology) of the $R$-submodule of $N_{\infty}$ generated by $\psi_j(1)$ for all $j \geq i$.
\end{lemma}

\begin{proof}
This follows immediately from the definitions and the computations in Lemma~\ref{between heads}.
\end{proof}

We are interested in identifying values of $i$ for which we can provide an explicit description of the submodule of $N_{\infty}$ described in Lemma~\ref{gens of kernel}.  The elements \[
\psi_i(1), \psi_{i+2}(1), \psi_{i+4}(1), \ldots \in N_{\infty} \cong R^{\wedge} \times R/1 \times R/3 \times \ldots
\]
have first components in $R^{\wedge}$ with $p$-adic valuations $a_i, a_{i+2}, a_{i+4}, \ldots$, respectively, where $a_m$ is defined in Definition~\ref{abdef}.  It is convenient to know when the initial value $a_i$ is the minimum of the set $\{a_i, a_{i+2}, a_{i+4}, \ldots\}$.  The following lemma (in which the parameters $g(n)$ from Definition~\ref{def gn} again appear) provides cases where this value $a_i$ is indeed the minimum.

\begin{lemma} \label{ain}
Let $i$ be an odd integer which is not divisible by $p$.  Let $a_j$ and $b_j$ be defined as in Definition~\ref{abdef}.  Assume there exists an odd integer $n$ such that $i < n$ and $a_i > a_n$.  Then there exists such an $n$ with $v_p(n) \geq 3$ and $b_{n-i} < v_p(n)$.  In particular, $n - i \leq g(n)$.
\end{lemma}

\begin{proof}
Choose a minimal $n$ such that $i < n$ and $a_i > a_{n}$.  In particular, we have $a_{n-2} > a_{n}$, so $v_p(n) \geq 3$.

There cannot exist odd $j$ such that $i < j < n$ and $v_p(j) \geq v_p(n)$.  Indeed, assume towards contradiction that there exists such a $j$.  Let $e = \max \{v_p(k) \colon i < k < n, \text{with $k$ odd}\}$ and let $j = \max \{ k \colon i < k < n, \text{ with $k$ odd and $v_p(k) = e$}\}$. We will show that $a_j \leq a_n$, which contradicts the minimality of $n$.

For $j$ and $n$ as defined above, by the inductive definition of the $a_k$ sequence, we have
\[
a_n - a_j = (2 - v_p(j+2)) + (2 - v_p(j+4)) + (2 - v_p(j+6)) + \cdots + (2 - v_p(j + n-j)).
\]
Observe that $j+2 < n$. Indeed, if $j+2=n$, then $\nu_p(2) =\nu_p(n-j) \geq \nu_p(n) \geq 3$ which is a contradiction. Next, by the definition of $j$, we know that $v_p(j) > v_p(k)$ for all odd $j+2 \leq k < n$.  This implies that $v_p(j+2) = v_p(2)$, $v_p(j+4) = v_p(4), \ldots, v_p(n-2) = v_p(n-j-2)$. We also have $v_p(n) \leq v_p(n-j)$.  In total, we have
\[
a_n - a_j \geq (2 - v_p(2)) + (2 - v_p(4)) + (2 - v_p(6)) + \cdots + (2 - v_p(n-j)).
\]
This latter expression is equal to $b_{n-j}$, and we have already seen that all terms in the $b_k$ sequence are non-negative.  Thus $a_n \geq a_j$, which is a contradiction.  We conclude that for all odd $j$ in the range $i < j < n$, we have $v_p(j) < v_p(n)$.

Write $\sum'$ to deduce the sum over the \emph{even} values in the specified interval.  Because $a_i > a_n$, we deduce that 
\[
n-i-\sideset{}{'}\sum_{k = 2}^{n-i} v_p(i+k) < 0.
\]
Rewriting,
\[
n-i-\sideset{}{'}\sum_{k = 2}^{n-i-2} v_p(n-k) < v_p(n).
\]
By our comments above, we have $v_p(n-k) = v_p(k)$ for all $k$ in the above range, so we have the inequality
\[
n-i-\sideset{}{'}\sum_{k = 2}^{n-i-2} v_p(k) < v_p(n).
\]
Because $i$ is not divisible by $p$, we also have that $n-i$ is not divisible by $p$, so the above inequality is the same as
\[
b_{n-i} < v_p(n).
\]
The result follows.
\end{proof}

\begin{proposition} \label{kernel gens}
For any odd $l \geq 1$, let $e_l \in N_{\infty} \cong R^{\wedge} \times R/1 \times R/3 \times \cdots$ be the element which is $1$ in the $R/l$ component and which is zero in all other coordinates.  Also write $e_{-1}$ for the element $(1, 0, 0, \ldots)$.
Let $Z_2$ be as in Definition~\ref{def of Z2} and assume $i \in Z_2$.  Let $A_i \in R$ be defined as in Definition~\ref{abdef}. Then the $R$-submodule of $N_{\infty}$ generated by 
\[
\psi_i(1), \psi_{i+2}(1), \psi_{i+4}(1), \ldots
\] 
is equal to the $R$-submodule generated by 
\[
A_ie_{-1}, e_{i}, e_{i+2}, e_{i+4}, \ldots.
\]
\end{proposition}

\begin{proof}
We show using induction that for any fixed even $j \geq 0$, the $R$-submodule of $N_{\infty}$ generated by 
\[
\psi_i(1), \psi_{i+2}(1), \psi_{i+4}(1), \ldots, \psi_{i+j}(1)
\]
is equal to the $R$-submodule generated by 
\[
A_ie_{-1}, e_{i}, e_{i+2}, e_{i+4}, \ldots, e_{i+j}.
\]
The claim of the proposition then follows immediately.

For the base case $j=0$, note that $\psi_i(1) = A_ie_{-1}$ (by Lemma~\ref{some Ni}, because $Z_2 \subseteq Z_1$) and that $e_i = 0 \in R/i$ (because $i \in Z_2$, and hence $i$ is not divisible by $p$).

Now assume the result has been shown for $j-2$.  We write
\[
\psi_{i+j}(1) = A_{i+j}e_{-1} + c_{i+j,1} e_1 + c_{i+j,3} e_3 + \cdots + c_{i+j, i+j - 2} e_{i+j-2} + e_{i+j}.
\]
We must have $c_{i+j,n} = 0$ for all $-1 < n < i$.  (If not, so $c_{i+j,n} \neq 0$ for some  $-1 < n < i$, then by Lemma~\ref{cin 0}, we have $i + j - n \leq g(n)$, but then $n < i \leq n + g(n)$, which contradicts our assumption that $i \in Z_2$.) 
So we have
\[
\psi_{i+j}(1) = A_{i+j}e_{-1} +c_{i+j,i}e_i +c_{i+j,i+2}e_{i+2}+ \cdots + c_{i+j, i+j - 2} e_{i+j-2} + e_{i+j}.
\]
We have that $A_{i+j}$ is a multiple of $A_i$ by Lemma~\ref{ain}, so $\psi_{i+j}(1)$ is in the $R$-submodule of $N_{\infty}$ generated by 
\[
A_ie_{-1}, e_{i}, e_{i+2}, e_{i+4}, \ldots, e_{i+j}.
\]
Conversely, by the induction hypothesis, we have that $A_ie_{-1}, e_{i}, \ldots, e_{i+j-2}$ are in the $R$-submodule generated by 
\[
\psi_i(1), \psi_{i+2}(1), \psi_{i+4}(1), \ldots, \psi_{i+j-2}(1),
\]  
and again note that by Lemma~\ref{ain}, the number $A_{i+j}$ is a multiple of $A_i$. Therefore 
\[
e_{i+j} = \psi_{i+j}(1) - A_{i+j}e_{-1} -c_{i+j,i}e_i -c_{i+j,i+2}e_{i+2}- \cdots - c_{i+j, i+j - 2} e_{i+j-2}
\]
is in the $R$-submodule generated by
\[
\psi_i(1), \psi_{i+2}(1), \psi_{i+4}(1), \ldots, \psi_{i+j}(1).
\]
This completes the induction.  
\end{proof}

\begin{proposition}
Let $i>1$ be an odd integer, and assume $i \in Z_2$, where $Z_2$ is the set defined in Definition~\ref{def of Z2}.  Then $K_{i-2}$ is isomorphic as an $R$-module to
\[
R^{\wedge} \times R/(i) \times R/(i+2) \times R/(i+4) \times \cdots.
\]
\end{proposition}

\begin{proof}
It's clear that the closure in $N_{\infty}$ of the $R$-submodule generated by $A_ie_{-1}, e_{i}, e_{i+2}, \ldots$ is isomorphic to the given module, so the claim follows from Proposition~\ref{kernel gens}.
\end{proof}

\begin{corollary} Let $i>1$ be an odd integer, and assume $i \in Z_2$, where $Z_2$ is the set defined in Definition~\ref{def of Z2}.  Then $N_{i-2}$ is isomorphic as an $R$-module to
\[
R/p^{a_i} \oplus R/1 \oplus R/3 \oplus \cdots \oplus R/(i-2).
\]

\end{corollary}

As a reality check we point out that this matches with Lemma \ref{some Ni} and the first short exact sequence on page 4.

\begin{example}
When $p = 3$, the set $Z_2$ from Definition~\ref{def of Z2} contains the following elements: 5, 7, 11, 13, 17, 19, 23, 31, 35, 37, 41, 43, 47, 49, 53, 55, 59, 61, 65, 67, 71, 73, 77, 85, 89, 91, 95, 97, 101, 103, 107, 109, 113, 115, 119, 121, 125, 127, 131, 139, 143, 145, 149, 151, 155, 157, 161, 163, 167, 169, 173, 175, 179, 181, 185, 193, 197, 199.

In other words, in addition to multiples of $3$, it only excludes the elements 25,  29, 79, 83, 133, 137, 187, 191. 
\end{example}

\begin{remark} It follows from Corollary \ref{Z-sequence} that all odd numbers of the form $\frac{p^a\pm 1}{2}$ for $a>0$ are contained in $Z_2$. 

\end{remark}

\section{The size of the sets $Z_1$ and $Z_2$}

We give lower-bounds on the proportion of positive odd integers in the sets $Z_1$ and $Z_2$ from Definition~\ref{def of Z1} and Definition~\ref{def of Z2}. 



The following lemma should be intuitively clear.  It roughly says that $2/p^e$ of odd integers are of the form $cp^e \pm b$, where $b$ is some fixed constant.  The point of the lemma is to give a precise version of this claim.

\begin{lemma} \label{proportion bounds}
Let $X$ denote the set of all odd positive integers, and for any positive number $N$, let $X_N := X \cap [0,N]$.  Fix a positive integer $e$, and an even positive integer $b$, where $b < \frac{p^e}{2}$.  Let
\[
Y_{e,N} := \{ x \in X_N \colon x = cp^e \pm b, \text{ some }c \in \ZZ\}.
\]
We have \[
\frac{\# Y_{e,N}}{\# X_N} < \frac{2}{p^e} + \frac{2}{\# X_N}.
\]
\end{lemma}

\begin{proof}
Let $Y_{e,N,-} := \{ x \in X_N \colon x = cp^e - b, \text{ some }c \in \ZZ\}.$  This set contains approximately half of the elements of $Y_{e,N}$, and it clearly suffices to prove that 
\[
\# Y_{e,N,-} < \frac{\# X_N}{p^e}  + 1.
\]
To ease notation, write $X$ instead of $X_N$ and write $Y$ for the set obtained from $Y_{e,N,-}$ obtained by removing its minimal element $p^e - b$ (we are ignoring the trivial case that $Y_{e,N,-}$ is the empty set).  In terms of this new notation, we must prove that 
\[
\# Y < \frac{\# X}{p^e}.
\]

Consider the following translates of $Y$: 
\[
Y, Y-2, Y-4, \ldots, Y - 2p^e + 2.
\]
We have listed $p^e$ of these sets, and they are pairwise disjoint. Their union is strictly contained in $X$ (because the minimal element in $Y$ is $3p^e - b$).  Because these sets all have cardinality $\# Y$, we deduce that $p^e \cdot \#Y < \# X$, as required.
\end{proof}

\begin{proposition}
Let $X$ denote the set of all odd positive integers.  Let $Z_1$ be as in Definition~\ref{def of Z1} and let $Z_2$ be as in Definition~\ref{def of Z2}.  Fix any positive integer $N$, and write $\lambda := \frac{2p-3}{2p-2}$.
We have
\[
\frac{\#(Z_1 \cap [0,N])}{\#(X \cap [0,N])} \geq 1 - \frac{1}{p} - \frac{1}{p^3} - \frac{1}{p^5} - \left(\sum_{k\geq 6 \textup{ even}} \frac{1}{p^{\lambda k}} \right) - \frac{\log_p (N)+1}{\lambda {\#(X \cap [0,N])} }.
\]
and
\[
\frac{\#(Z_2 \cap [0,N])}{\#(X \cap [0,N])} \geq 1 - \frac{1}{p} - \frac{2}{p^3} - \frac{2}{p^5} - \left(\sum_{k\geq 6 \textup{ even}} \frac{2}{p^{\lambda k}} \right) - \frac{2\log_p (N)+2}{\lambda {\#(X \cap [0,N])}}.
\]
\end{proposition}

\begin{proof}
We prove the result for $Z_2$.  The argument for $Z_1$ is the same.

To obtain $Z_2$, we can proceed as follows:
\begin{enumerate}
\item Begin with the set $X_N$ of all odd positive integers up to $N$.
\item Remove all multiples of $p$.
\item For every even positive integer $k$, let $e_k$ be the minimal positive integer such that $g(p^{e_k}) \geq k$.  Remove all elements of the form $cp^{e_k} \pm k$, i.e., the set $Y_{e_k, N}$.
\end{enumerate}
The second step removes at most $\frac{\# X_N}{p}$ many elements. For the third step, we explicitly know that $e_2 = 3$ and $e_4 = 5$.  The values $e_j$ for $j \geq 6$ depend on the odd prime $p$, and instead of finding these values exactly, we use the bound from Lemma~\ref{estimate gpa}.  From that lemma's proof, we have 
\[
e_k > \frac{2p-3}{2p-2} \, k.
\]
We can then apply Lemma~\ref{proportion bounds} to estimate
\[
  \frac{Y_{e_k, N}}{\# X_N}<\frac{2}{p^{e_k}} + \frac{2}{\# X_N} .
\] 
While the first summand assembles into a geometric series, the second one is constant in $X_N$. This requires us to bound the occurrence of $k$ such that $Y_{e_k, N}$ is nonempty. The smallest possible element in $Y_{e_k, N}$ is given by $p^{e_k}-k$, which is bounded below by
\[
  p^{\frac{2p-3}{2p-2} \, k}-k.
\]
We denote $\lambda=\frac{2p-3}{2p-2}$, and claim that for $\lambda k>\log_p(N)+1$, the inequality $p^{\lambda k}-k>N$ holds as desired:
For this, note first that $N=p^{\log_p N +1 -1}$, so it suffices to show that 
\[
  p^{\lambda k}-k>p^{\lambda k-1} \Leftrightarrow (p-1)p^{\lambda k -1}-k>0.
\]
As $k$ ranges over even positive integers, we first check this manually for $k=2$: 
\[
  (p-1)p^{\lambda 2 -1}-2 =(p-1)p^{\frac{(2p-3)-(p-1)}{p-1}}-2=\overbrace{(p-1)}^{\geq 2}\overbrace{p^{\frac{p-2}{p-1}}}^{>1}-2>0.
\]
Next, note that for $f(k)=(p-1)p^{\lambda k -1}-k$, we have the derivative $f'(k)=(p-1)\log(p)p^{\lambda k -1}-1$. By the same reason as above, this is non-negative for $k\geq 0$, so $f(k)>0$ for $k\geq 2$. 

The upshot is then that we only need to remove the sets $Y_{e_k,N}$ for $k\leq \frac{\log_p(N)+1}{\lambda}$. Putting together our observations, this leaves us with the result:
\begin{align*}
\frac{\#(Z_2 \cap [0,N])}{\#(X \cap [0,N])} & \geq 1 - \frac{1}{p} - \sum_{k\geq 2 \textup{ even}}\frac{\# Y_{e_k,N}}{\# X_N}\\
&> 1 - \frac{1}{p} - \frac{2}{p^3} - \frac{2}{p^5} - \left(\sum_{k\geq 6 \textup{ even}} \frac{2}{p^{\lambda k}} \right) - \frac{2\log_p (N)+2}{\lambda \#X_N }.
\end{align*}
\end{proof}

\begin{remark}
The proportion of elements in $Z_1$ and $Z_2$ is larger when the (odd) prime $p$ is larger.  For example, when $p = 3$, we have  
\[
\liminf \frac{\#(Z_1 \cap [0,N])}{\#(X \cap [0,N])} \geq 0.61 \text{ and } \liminf \frac{\#(Z_2 \cap [0,N])}{\#(X \cap [0,N])} \geq 0.58,
\]
and when $p = 101$, we have
\[
\liminf \frac{\#(Z_1 \cap [0,N])}{\#(X \cap [0,N])} \geq 0.99 \text{ and } \liminf \frac{\#(Z_2 \cap [0,N])}{\#(X \cap [0,N])} \geq 0.99.
\]
\end{remark}

\bibliography{Tate}
\bibliographystyle{plain}

\end{document}